\documentclass[a4paper, reqno, 12pt]{amsart}
\usepackage[english]{babel}
\usepackage[T1]{fontenc}
\usepackage{mathptmx,amsmath,dsfont}
\usepackage{hyperref}
\usepackage{graphicx}
\usepackage{amssymb,verbatim,cases,bbold}
\usepackage{fullpage}
\usepackage[svgnames,hyperref]{xcolor}
\usepackage{color,soul}
\newtheorem{theorem}{Theorem}[section]
\newtheorem{corollary}[theorem]{Corollary}
\newtheorem{lemma}[theorem]{Lemma}
\newtheorem{definition}[theorem]{Definition}

\newtheorem{remark}[theorem]{Remark}

\numberwithin{equation}{section}
\numberwithin{equation}{section}

\newcommand{\C}{\mathbb{C}}

\newcommand{\R}{\mathbb{R}}
\newcommand{\N}{\mathbb{N}}

\newcommand{\E}{\mathcal{E}}

\newcommand{\F}{\mathcal{F}}

\newcommand{\NP}[1]{\text{NP}({dd^c#1})^n}
\newcommand{\NNP}{\mathcal{N}_\text{NP}}

\newcommand{\Om}{\Omega}

\newcommand{\pshn}{\text{PSH}^-(\Omega)}

\newcommand{\ddc}{dd^c}
\newcommand{\psh}{\text{PSH}(\Omega)}
\begin{document}
\author{Thai Duong Do\textit{$^{1}$}, Ngoc Thanh Cong Pham\textit{$^{2}$}}
\address{\textit{$^{1}$}Institute for Artificial Intelligence, University of Engineering and Technology, Vietnam National University, Hanoi, Vietnam.}
\email{dtduong@vnu.edu.vn, duongdothai.vn@gmail.com}
\address{\textit{$^{2}$} Institute of Mathematics, Vietnam Academy of Science and Technology, 18, Hoang Quoc Viet, Hanoi, Viet Nam}
\email{ngocthanhcong.pham@gmail.com, cong.pnt.math@gmail.com}

\title{An existence theorem for non-pluripolar complex Monge-Amp\`ere type equations on hyperconvex domains}

\date{\today}
\maketitle
\begin{abstract}
In this paper, we study the non-pluripolar complex Monge-Amp\`ere measure on bounded domains in \( \mathbb{C}^n \). We establish a general existence theorem for a non-pluripolar complex Monge-Amp\`ere type equation with prescribed singularity on a bounded hyperconvex domain in \( \mathbb{C}^n \).
\end{abstract}

\tableofcontents

\section{Introduction}
Let $\Omega$ be a bounded domain in $\C^n$. We write $d=\partial+\overline{\partial}$ and $d^c=i(\overline{\partial}-\partial)$. Denote by $\psh$ and $\pshn$ respectively the set of
 plurisubharmonic (psh) functions and the set of negative psh functions on $\Om$. The complex Monge-Amp\`ere operator of a smooth psh function $u$ on $\Om$  is a regular Radon measure given by 
	$$(dd^c u)^n =c_n \mbox{det}\begin{bmatrix}
			\frac{\partial^2 u}{\partial z_j\partial \bar{z}_k}
		\end{bmatrix} dV,$$ 
	where $c_n>0$ is a constant depending only on $n$ and $dV=\Big(\frac{i}{2}\Big)^n\prod_{j=1}^ndz_j\wedge d\overline{z}_j$ is the standard volume form.

A natural question that arises is to determine the domain of the Monge-Amp\`ere operators. 
In \cite{BT87}, Bedford and Taylor have investigated the plurifine topology, originally introduced by Fuglede in \cite{Fu86} as the weakest topology in which all psh functions are continuous. In \cite{BT82,BT87}, they have defined the non-pluripolar complex Monge-Amp\`ere measure, often referred to as the ``non-pluripolar part of the Monge-Amp\`ere operator'', for every psh function. For a negative psh function \( u \), the non-pluripolar Monge-Amp\`ere measure of $u$ is defined as follows
$$ \NP{u}=\lim_{M\rightarrow\infty} \mathbb{1}_{\{u > -M\}}(dd^c \max\{u, -M\})^n. $$ This yields a Borel measure that does not charge pluripolar sets and can have locally unbounded mass. Moreover, they have defined $(dd^cu)^n$ as a Radon measure in the case where $u$ is a continuous psh function and then in the case where $u$ is a locally bounded psh function.


In the global setting, for a compact complex K\"ahler manifold \((X, \omega)\), based on the idea of the non-pluripolar complex Monge-Amp\`ere measure in the local setting, in \cite{GZ07}, Guedj and Zeriahi have introduced the non-pluripolar complex Monge-Amp\`ere measure \((\omega+dd^c u)^n\) for any $\omega$-psh function $u$ as follows
$$(\omega+dd^c u)^n=\lim_{M\rightarrow\infty} \mathbb{1}_{\{u > -M\}}(\omega+dd^c \max\{u, -M\})^n. $$ Since the total mass of the measures $\mathbb{1}_{\{u > -M\}}(\omega+dd^c \max\{u, -M\})^n$ 
is uniformly bounded from above by $\int_X\omega^n$ and by Stokes theorem, $(\omega+dd^c u)^n$ is a well-defined finite positive Borel measure on $X.$
In \cite{BEGZ10}, for a smooth, closed real \((1,1)\)-form \(\theta\) on \(X\) whose cohomology class \(\alpha\) is big, Boucksom, Eyssidieux, Guedj, and Zeriahi have extended the results obtained in \cite{GZ07} to the case of arbitrary cohomology classes, and introduced the non-pluripolar complex Monge-Amp\`ere measure \((\theta + dd^c u)^n\) for any \(\theta\)-psh function \(u\). In \cite{DDL18, DDL21}, Darvas, Di Nezza, and Lu have studied the complex Monge-Amp\`ere equation with prescribed singularity types. They have introduced the concepts of model potentials and model-type singularities, and demonstrated that the equation is well-posed only when the potential has model-type singularities.

In local setting, when restricting attention to only hyperconvex domains, Cegrell has introduced finite energy classes of psh functions which are now known as Cegrell's classes (see \cite{Ceg04}). He has also proved that the class $\E$  is the largest subclass of the set of negative psh functions for which the complex Monge-Amp\`ere measure is a well-defined, regular Radon measure. After that, the domain of the Monge-Amp\`ere operators has been well understood after the work of B\l ocki (\cite{B06}). Recently, Cegrell's classes have been studied by many authors. For more details we refer the readers to \cite{Aha07,ACCP,ACH07,Ben09,Ben14,Ben11,BGZ,DD19,DDP20,NP09}.

In \cite{DDLP}, inspired by \cite{DDL18, DDL21},  we (together with Do and Le) have studied the non-pluripolar Monge-Amp\`ere measure for psh functions in a bounded domain in $\C^n$. We have introduced the notation of model psh functions in local setting: a function $u\in\pshn$ is called  model if $u=P[u],$ where
	\begin{multline*}
	P[u]=\big(\sup\{v\in\pshn:\ v\leq u+ O(1) \text{ on }\Omega,\ \liminf\limits_{\Omega\setminus N\ni z\rightarrow\xi_0}(u(z) -v(z))\geq0\ \forall\xi_0\in\partial\Omega \}\big)^*,
	\end{multline*}
	and $N=\{u=-\infty\}$. 
We have also studied a Dirichlet type problem for the non-pluripolar complex Monge-Amp\`ere equation  with prescribed singularity.

In this paper, we study the non-pluripolar complex Monge-Amp\`ere measure on bounded domains in $\C^n$. We also study the following Dirichlet type problem for the non-pluripolar complex Monge-Amp\`ere equation in bounded hyperconvex domains:
\begin{equation}\label{NPMAE}
		\begin{cases}
		\NP{u}=F(u, \cdot)d\mu,\\
		P[u]=\phi.
		\end{cases}
		\end{equation}
Denote by $\mathcal{N}_{NP}(\Om)$ (or $\mathcal{N}_{NP}$ for short) the set of negative psh functions $u$ on $\Om$ with smallest model psh majorant identically zero (i.e., $P[u]=0$).
 For every $H\in\pshn$, we denote
	$$\mathcal{N}_{NP}(H)=\{w\in\pshn: \mbox{there exists } v\in\mathcal{N}_{NP} \mbox{ such that } v+H\leq w\leq H \}.$$
Our main result is as follows.
\begin{theorem}\label{Thm L1}
    Let $\Omega$ be a bounded hyperconvex  domain in $\mathbb{C^n}$ and
$\mu $ be a nonnegative Borel measure that  vanishes on all pluripolar subsets of $\Omega$. Let $\phi$ be a model psh function on $\Om$. Assume that $F : \R \times \Om\to [0, +\infty )$ is a measurable function satisfying the following conditions:
\begin{itemize}
    \item[1)] For all $z\in \Om $ the function $t \mapsto F(t, z)$ is continuous and  nondecreasing;
\item[2)] For all $t\in \R $, the function $z\mapsto F(t,z) $ belongs to $L^1_{loc}(\Om, \mu)$;
\item[3)] There exists $\psi \in \NNP$ such that $$\NP{\psi} \geq F(\phi, \cdot)d\mu.$$ 
\end{itemize}
Then there exists uniquely determined function $u \in \NNP(\phi)$ that solves the non-pluripolar complex Monge-Amp\`ere equation (\ref{NPMAE}).
    \end{theorem}
\noindent Our result may be interpreted as a non-pluripolar Monge-Amp\`ere counterpart to the main theorem established by Benelkourchi in \cite{Ben14}, which was developed in the context of Cegrell’s classes. Our proof of the main theorem is divided into two parts, depending on whether \(F\) belongs to \(L^1\) or to \(L^1_{\text{loc}}\). When \(F \in L^1\), we follow an idea from Benelkourchi’s method: we construct a convex, compact, and bounded set \(\mathcal{A} \subset L^1\), and define a continuous operator \(T\) on \(\mathcal{A}\) so that Schauder’s fixed point theorem applies. In the case \(F \in L^1_{\text{loc}}\), we consider an appropriate exhaustion of \(\Omega\) by an increasing sequence of subdomains and solve the problem on each subdomain. The solution on \(\Omega\) is then obtained as the limit of the solutions on these subdomains.
\section{Non-pluripolar complex Monge-Amp\`ere measures}

First, we recall the definition of the non-pluripolar complex Monge-Amp\`ere measures.
	\begin{definition}\label{def NPMA}\cite{BT87}
		If $u\in\psh$ then the non-pluripolar complex Monge-Amp\`ere measure of $u$ is the measure $\NP{u}$ satisfying 
		$$\int\limits_E\NP{u}=\lim\limits_{j\rightarrow\infty}\int\limits_{E\cap\{u>-j\}}\left(\ddc \max\{u,-j\} \right)^n,$$
		for every Borel set $E\subset\Om.$
	\end{definition}
We recall and prove several results concerning the non-pluripolar Monge-Amp\`ere measure, which will be used in sequel.
	
	\begin{remark}\label{remk NPMA}
		\begin{itemize}
			\item[i.]If $E\subset\{u>-k\},$ then it follows from \cite[Corollary 4.3]{BT87} that $$\int\limits_E(\ddc\max\{u, -j\})^n=\int\limits_E(\ddc\max\{u,-k\})^n, \text{ for every }j\geq k.$$
			In particular, 
			$$\int_{E\cap\{u>-k\}}\NP{u}=\int\limits_{E\cap\{u>-k\}}(\ddc\max\{u,-k\})^n,$$
			for every $k>0$ and for every Borel set $E\subset\Om$.
			\item[ii.] $\NP{u}$ vanishes on every pluripolar sets.
		\end{itemize}
	\end{remark}

\begin{lemma}\label{NPMA max}
    Let $u, v \in \pshn$, then 
    $$\NP{\max(u, v)} \geq \mathbb{1}_{\{u \geq v\}}\NP{u}+\mathbb{1}_{\{u < v\}}\NP{v}.$$
\end{lemma}

\begin{proof}
			Since non-pluripolar complex Monge-Amp\`ere measure does not charge the set $\{ u + v = -\infty\}$, it suffices to prove that
				$$\int\limits_{E} \NP{\max \{ u,v\}} \geq \int\limits_{E \cap \{u \geq v\}} \NP{u}+\int\limits_{E \cap \{u < v\}} \NP{v},$$
			 for every Borel set $E \subset \{ u+ v >-\infty\}$. Observe that $E = \bigcup\limits_{j \geq 1} E_j$ where $E_j = E \cap \{ u + v >-j\}$. We aim to show that
			\[
			\int\limits_{E_{j_0}} \NP{\max \{ u,v\}} \geq \int\limits_{E_{j_0} \cap \{u \geq v\}} \NP{u}+\int\limits_{E_{j_0} \cap \{u < v\}} \NP{v},
			\]
			for every $j_0\geq 1$.
			
			Since $\max \{u,v\} \geq \min\{u, v\}\geq u+v$,  it follows that
			$$E_{j_0}\subset\{u+v > -j_0\}\subset \{u>-j\}\cap\{v>-j\} \subset \{\max \{u,v\} > -j \},$$ 
			for every $j > j_0$.  Hence, by Definition \ref{def NPMA} and Remark \ref{remk NPMA} (i), we have
		\begin{equation}\label{eq1lemNPmax}
			\int\limits_{E_{j_0}} \NP{w} = \int\limits_{E_{j_0} } \Big(\ddc \max\{ w, -j \} \Big)^n,
		\end{equation}
			for $w\in\{u, v, \max\{u, v\}\}$ and for every $j>j_0$.
			
			Denote $u_j=\max\{u, -j\}$, $v_j=\max\{v, -j\}$ and $\phi_j=\max\{\max\{u, v\}, -j\}$.
		Observe that $\phi_j = \max \{u_j, v_j\}$. By applying 
			\cite[Proposition 11.9]{Dem89}
			 (see also \cite[Proposition 4.3]{NP09}), we have
			 \begin{equation}\label{eq1.1lemNPmax}
			 (\ddc\phi_j)^n\geq \mathbb{1}_{\{u_j\geq v_j\}}(\ddc u_j)^n+\mathbb{1}_{\{u_j< v_j\}}(\ddc v_j)^n.
			 \end{equation}
			Note $E_{j_0}\cap \{u_j\geq v_j\}=E_{j_0}\cap \{u\geq v\}$ and  $E_{j_0}\cap \{u_j<v_j\}=E_{j_0}\cap \{u< v\}$ which hold for every
			$j>j_0$. Hence, it follows from \eqref{eq1.1lemNPmax} that
			 \begin{equation}\label{eq2lemNPmax}
			 \int\limits_{E_{j_0} } (\ddc\phi_j)^n 
			 \geq \int\limits_{E_{j_0}\cap \{u\geq v\}}(\ddc u_j)^n 
			 +  \int\limits_{E_{j_0}\cap \{u<v\}} (\ddc v_j)^n,
			 \end{equation}
			for every $j>j_0$.
			
			Combining \eqref{eq1lemNPmax} and \eqref{eq2lemNPmax}, we get 
			$$\int\limits_{E_{j_0}} \NP{\max\{u, v\}}\geq
			 \int\limits_{E_{j_0}\cap \{u\geq v\}} \NP{u}+\int\limits_{E_{j_0}\cap \{u<v\}} \NP{v},$$
		which establishes the desired result.
		\end{proof}



\begin{lemma}\label{NNP is convex}
    Let $u, v \in \NNP$ and $t \in (0, 1)$. Then $tu  + (1-t)v \in \NNP$.
\end{lemma}

\begin{proof} We define the singular sets $N_u=\{u=-\infty\}$, $N_v=\{v=-\infty\}$, and set
$$U=\{w\in\pshn:\ w\leq u+ O(1) \text{ on }\Omega,\ \liminf\limits_{\Omega\setminus N_u\ni z\rightarrow\xi_0}(u(z) -w(z))\geq0\ \forall\xi_0\in\partial\Omega \}$$
and
$$V=\{w\in\pshn:\ w\leq v+ O(1) \text{ on }\Omega,\ \liminf\limits_{\Omega\setminus N_v\ni z\rightarrow\xi_0}(v(z) -w(z))\geq0\ \forall\xi_0\in\partial\Omega \}.$$
    By the definition of $P[u]$ and $P[v]$, there exist sequences $\{u_j\}_j \subset U$ and $\{v_k\}_k \subset V$ such that 
    $$\lim\limits_{j \to +\infty}u_j=P[u] \quad\text{ and } \quad\lim\limits_{k \to +\infty}v_k= P[v].$$ 
    Since $u_j \leq u + O(1)$ and $v_k \leq v + O(1)$ on $\Omega$, then for all $t \in (0, 1)$, we have
    \begin{equation}\label{tuj+(1-t)vk<tu+(1-t)v+O}
        tu_j + (1 - t)v_k \leq tu + (1 - t)v + O(1) \quad \text{on } \Omega.
    \end{equation}
    It is clear that 
    $$\{tu+(1-t)v\ \neq -\infty\} \subset \{u\neq -\infty\}\cap\{v \neq -\infty\},$$ 
    for every $t\in (0, 1)$. Thus,
    \begin{align}\label{condition liminf}
        \liminf\limits_{\Omega\setminus N\ni z\rightarrow\xi_0}&(tu+(1-t)v-tu_j-(1-t)v_k)\nonumber\\ 
        &\geq  t\liminf\limits_{\Omega\setminus N\ni z\rightarrow\xi_0}(u-u_j)+(1-t)\liminf\limits_{\Omega\setminus N\ni z\rightarrow\xi_0}(v-v_k)\nonumber\\
        &\geq t\liminf\limits_{\Omega\setminus N_u\ni z\rightarrow\xi_0}(u-u_j)+(1-t)\liminf\limits_{\Omega\setminus N_v\ni z\rightarrow\xi_0}(v-v_k)\nonumber\\
        &\geq 0,
    \end{align}
    where $N=\{tu+(1-t)v=-\infty\}$ and $\xi_0\in\partial\Omega$.
    It follows from inequalities \eqref{tuj+(1-t)vk<tu+(1-t)v+O} and \eqref{condition liminf} that $tu_j+(1-t)v_k$ belongs to the set
    $$\{w \in \pshn: \ w\leq tu+(1-t)v+ O(1) \text{ on }\Omega,\ \liminf\limits_{\Omega\setminus N\ni z\rightarrow\xi_0}(tu+(1-t)v -w)\geq0\ \forall\xi_0\in\partial\Omega \}.$$
    Hence,
    $$P[tu+(1-tv)] \geq tu_j+(1-t)v_k \text{ on } \Om.$$
    Letting $j, k \to +\infty$, we obtain
    $$P[tu+(1-tv)] \geq tP[u]+(1-t)P[v]=0.$$
    However, the definition implies that $P[tu + (1 - t)v] \leq 0$. Thus, we conclude that
    \[
    P[tu + (1 - t)v] = 0,
    \]
    which means  $tu + (1 - t)v \in \NNP$, as desired.
\end{proof}

The following corollary follows directly from the proof of the preceding lemma.
\begin{lemma}
    If $u, v \in \pshn$, then $P[u+v] \geq P[u]+P[v]$.
\end{lemma}


This section concludes with a corollary of \cite[Theorem 5.6]{DDLP}, which can be regarded as a stronger version of  \cite[Corollary 5.7]{DDLP}.
\begin{corollary}\label{cor: comparison wrt NPMA 2}
    Let $H\in\pshn$ and $u,v\in\NNP(H)$. Assume that $\NP{u}\leq\NP{v}$ on the set $\{u < v\}$. Then $u\geq v.$ 
\end{corollary}

\begin{proof}
    Let $w_j\in\text{PSH}(\Omega,[-1, 0]),$ $j=1,...,n,$ and denote $T=dd^cw_1\wedge...\wedge dd^c w_n$. Then, by  \cite[Theorem 5.6]{DDLP} and the assumption $\NP{u}\leq\NP{v}$ we known that
    $$\frac{1}{n!}\int\limits_{\{u<v\}}(v-u)^n T\leq 0.$$
    Hence, $u \geq v$ almost everywhere in $\Om$, which means $u \geq v$ everywhere in $\Om$, as desired.
\end{proof}

\section{Proof of the main result}
For the reader's convenience, we recall the main result from \cite{DDLP}, which plays a crucial role in the proof of our main result.
\begin{theorem}\label{MA w prescribed singularity}\cite[Theorem 1.2]{DDLP} Let $\Om$ be a bounded domain, $\mu$  is a non-pluripolar positive Borel measure on $\Om$ and $\phi$  is a model psh
 function on $\Om$.
	Assume that there exists $v\in\pshn$ such that $\NP{v}\geq\mu$ and $P[v]=\phi$. Denote 
	$$S=\{w\in\pshn: w\leq  \phi, \NP{w}\geq\mu\}.$$
	Then $u_S:=(\sup\{w: w\in S\})^*$ is a solution of the problem 
		\begin{equation}\label{NPMA1}
		\begin{cases}
		\NP{u}=\mu,\\
		P[u]=\phi,
		\end{cases}
		\end{equation}
	 Moreover, if there exists $\psi\in\NNP$ such that 
	$\NP{\psi}\geq\mu$ then $u_S$ is the unique solution of \eqref{NPMA1}.
\end{theorem}
The following lemma, which can be considered as a non-pluripolar Monge-Amp\`ere counterpart to \cite[Lemma 3.1]{Ben14}, will be used in the proof of our main result.
\begin{lemma}(Stability)\label{stability}
Let $\Omega$ be a bounded hyperconvex domain in $\C^n$, $\mu $ be a finite nonnegative measure that  vanishes on all pluripolar subsets of $\Om$ and $\phi$ be a model psh  function on $\Om$. Fix a function
 $v_0 \in \NNP(\phi)$ such that $\displaystyle\int\limits_\Om\NP{v_0}<+\infty$. Then for any $u_j ,\ u \in \NNP(\phi)$
solving
$$\begin{cases}
    \NP{u_j}=h_jd\mu,\\
    P[u_j]=\phi
\end{cases} \hspace{0.5cm}
\text{ and } \hspace{0.5cm} \begin{cases}
    \NP{u}=hd\mu,\\
    P[u]=\phi
\end{cases}$$
such that $0\le h d\mu , h_j d\mu\le \NP{v_0} $
 and $h_j$ converges $\mu$-a.e to $h$, we have that $u_j $ converges  to $u$ in $L^1(\Om)$.
\end{lemma}

\begin{proof}
It follows from \cite[Corollary 5.7]{DDLP} that
$u_j\ge v_0,\ \forall j\in \N.$ Thus, the sequence $(u_j)_j$ is
 relatively compact in the  $L^1_{loc}-$topology.
Let $\tilde u \in \NNP(\phi) $ be any closter point of the
 sequence $\{u_j\}_j$. Assume that $u_j \to \tilde u$
pointwise almost everywhere with respect to the Lebesgue measure $dV$. Then, by Lemma 2.1 in
\cite{Ce1}, after passing to a subsequence if necessary, we may assume $u_j \to \tilde u $ almost everywhere with respect to $d\mu$.
It follows that
$$
\tilde u = (\limsup_{j\to +\infty}u_j )^* = \lim_{j\to +\infty} (\sup_{k\ge j}u_k)^*.
$$
We set
$$
\tilde u_j= (\sup_{k\ge j}u_k)^* =
(\lim_{l\to +\infty} \sup_{l\ge k \ge j}u_k)^*=(\lim_{l\to +\infty} (\sup_{l\ge k \ge j}u_k)^*)^* =( \lim _{l\to +\infty}
\tilde{u}_j^l)^*,
$$
where, $\tilde{u}_j^l=(\sup\limits_{l\ge k \ge j}u_k)^*$.
Fix $j, l \in \N$ such that $l \geq j$. Then for every $k \in [j, l]$,
$$\NP{u_k}=h_kd\mu \geq \min\limits_{l \geq k \geq j}h_kd\mu.$$
Therefore, applying \cite[Lemma 4.2]{DDLP}, we obtain
$$\NP{\tilde{u}_j^l}\geq \min\limits_{l \geq k \geq j}h_kd\mu, \quad\text{ for each } l>j.$$
Again, using \cite[Lemma 4.2]{DDLP},  we deduce that for each $j \in \N$,
$$\NP{\tilde{u}_j} \geq \lim\limits_{l \to +\infty}\min\limits_{l \geq k \geq j}h_kd\mu.$$
Applying  \cite[Lemma 3.2]{DDLP}, we conclude that
$$\NP{\tilde{u}}\geq \lim\limits_{j \to +\infty} \left( \lim\limits_{l \to +\infty}\min\limits_{l \geq k \geq j}h_kd\mu\right) = hd\mu.$$
To prove the reverse inequality, let $w \in \text{PSH}(\Omega,[-1, 0])$ be any test function. Since 
$u_j \le \tilde u$,  , it follows that for every $\epsilon>0$, the set $\{u_j < \tilde u+\varepsilon\}=\Om$. Hence, applying \cite[Theorem 5.6]{DDLP}, we obtain
$$
\int\limits_{\Om} (-w) \NP{u_j}  \ge \int\limits_{\Om} (-w) \NP{\tilde{u}}.
$$
Therefore
$$
\lim_{j\to +\infty}\int\limits_{\Om} w h_j d\mu
 \leq \int\limits_{\Om}  w\NP{\tilde{u}}.
$$
Combining this with the previous inequality, we conclude that for every test function $w \in \text{PSH}(\Omega,[-1, 0])$,
$$
 \int\limits_{\Om}  w \NP{\tilde{u}} =  \int\limits_{\Om}  w h d\mu.
$$
Since any test function can be expressed as the difference of two bounded psh functions
(cf Lemma 2.1 in \cite{Ceg04}) this equality holds for any test function $w$. Hence
$$
\NP{\tilde{u}} = hd\mu = \NP{u}.
$$
Finally, by \cite[Corollary 5.7]{DDLP}, it follows that $\tilde{u}=u$ which completes the proof.
\end{proof}

\begin{proof}[Proof of Theorem~\ref{Thm L1}]
    We first assume that $F(t, \cdot) \in L^1(\Om, \mu)$, so that $F(\phi, \cdot) \in L^1(\Om, \mu)$ as well.  Since $F(\phi, \cdot)d\mu$ is a non-pluripolar positive Borel measure on $\Om$, it follows from Theorem $\ref{MA w prescribed singularity}$ that the function
    $$u_0=\left( \sup\{w \in \pshn: w \leq \phi, \NP{w} \geq F(\phi, \cdot)\mu\} \right)^*$$
    satisfies
    \begin{equation*}
		\begin{cases}
		\NP{u_0}=F(\phi, \cdot)\mu,\\
		P[u_0]=\phi.
		\end{cases}
    \end{equation*}
    Moreover, Theorem $\ref{MA w prescribed singularity}$ also asserts that $u_0 \in \NNP(\phi)$.

We define the following set
$$\mathcal{A}=\{ u \in \NNP(\phi): u \geq u_0\}.$$
Since $u_0 \in \mathcal{A}$, we have $\mathcal{A} \neq \emptyset$. It follows from \cite[Theorem 3.2.12]{Ho} that $\mathcal{A}$ is compact in $L^1_{\text{loc}}(\Om, dV)$.
We now show that  $\mathcal{A}$ is convex. Indeed, let $w_1, w_2 \in \mathcal{A}$ and $t \in (0, 1)$. By the definition of $w_1$ and $w_2$, there exists $v_1, v_2 \in \NNP$ such that 
$$v_j + \phi \leq w_j \leq \phi, \quad \text{for  } j =1, 2.$$
It follows that $$tv_1+(1-t)v_2 + \phi \leq tw_1+(1-t)w_2 \leq \phi.$$
By Lemma \ref{NNP is convex} that $tv_1+(1-t)v_2 \in \NNP$. Consequently, $tw_1+(1-t)w_2 \in \NNP(\phi)$, which shows that $\mathcal{A}$ is convex, as desired.\\
Next, for each $u \in \mathcal{A}$, we define the map $T : \mathcal{A} \to \mathcal{A}$ by $u \mapsto \hat{u}$, where
$$\hat{u}=\left( \sup\{w \in \pshn: w \leq \phi,\ \NP{w} \geq F(u, \cdot)d\mu\} \right)^*.$$
We claim that $T$ is well-defined. Indeed, it follows from the first and the third assumptions on $F$ that there exists   $\psi \in \NNP$ such that
$$\NP{\psi} \geq F(\phi, \cdot)d\mu \geq F(u, \cdot)d\mu.$$
Thus, by Theorem $\ref{MA w prescribed singularity}$, the function $\hat{u}$  is the unique solution to the following equation
    \begin{equation*}
		\begin{cases}
		\NP{\hat{u}}=F(u, \cdot)\mu,\\
		P[\hat{u}]=\phi.
		\end{cases}
    \end{equation*}
Since $u \in \NNP(\phi)$, we have $u \leq \phi$. Moreover, because  $F$ is nondecreasing in the first argument, we obtain
$$\NP{\hat{u}}=F(u, \cdot)\mu \leq F(\phi, \cdot)\mu=\NP{u_0}.$$
Furthermore, since both $\hat{u}$ and $u_0$ belong to $\NNP(\phi)$, it follows from \cite[Corollary 5.7]{DDLP} that $\hat{u} \geq u_0$. Hence, $\hat{u} \in \mathcal{A}$, and so $T$ is well-defined.\\
We also claim that $T$ is continuous. To prove this, let $\{u_j\}_j \subset \mathcal{A}$ be a sequence converging  to some $u \in \mathcal{A}$. By Lemma \ref{stability}, it suffices to verify that
$$F(u_j , .) d\mu \to F(u, .)d \mu .$$ 
After passing to a subsequence if necessary, we may assume that
$u_j \to u $ almost everywhere with respect to $dV $. By \cite[Lemma 2.1]{Ce1},  we also have $u_j \to u $
$d\mu$-almost everywhere. Then, by the Lebesgue dominated convergence theorem, we conclude that
$F(u_j , \cdot) d\mu \to F(u, \cdot)d \mu .$\\
By Schauder-Tychonoff's fixed point theorem, there exists $u \in \mathcal{A}$ such that $T(u)=u$. This means that $u$ is a solution to the problem \eqref{NPMAE}. The uniqueness of $u$ follows from \cite[Corollary 5.7]{DDLP}.

We now proceed to the proof of the general case. For each $j \geq 1$, we define
$$\Omega_j=\{z \in \Om: d(z, \partial\Om)>2^{-j}, \phi+\psi>-2^{-j}\}.$$
Then $\{\Om_j\}_j$ is an increasing sequence of plurifine open subsets of $\Om$ and $\Om_j \Subset \Om$ for all $j \geq 1$. For each $j \in \N^*$, we set $\mu_j=\mathbb{1}_{\Om_j}\mu$. By the special case already treated, for each $j$, there exists a uniquely determined function $u_j \in \NNP(\phi)$ solving the non-pluripolar complex Monge–Ampère equation
\begin{equation}
		\begin{cases}
		\NP{u_j}=F(u_j, \cdot)d\mu_j,\\
		P[u_j]=\phi.
		\end{cases}
		\end{equation}
        For every $j \geq 1$, on the set $\{u_j < u_{j+1}\}$, we have
        $$\NP{u_j}=F(u_j, \cdot)d\mu_j \leq F(u_{j+1}, \cdot)d\mu_j\leq F(u_{j+1}, \cdot)d\mu_{j+1}=\NP{u_{j+1}}.$$
        It thus follows from Corollary \ref{cor: comparison wrt NPMA 2} that $u_j \geq u_{j+1}$, so the sequence $\{u_j\}_j$ is decreasing.\\
         We now define  $u=\lim\limits_{j \to \infty}u_j$. By using the monotonicity of $F$ and  \cite[Lemma 2.6]{DDLP}, we have
        $$\NP{u_j}=F(u_j, \cdot)d\mu\leq F(\phi, \cdot)d\mu\leq\NP{\psi}\leq \NP{(\psi+\phi)},$$
        for all $j \geq 1$. Therefore, by  \cite[Corollary 5.7]{DDLP}, we conclude that
        \begin{equation}\label{phi+ psi< u}
            \phi+\psi \leq u_j, \quad\text{ for all } j \geq 1.
        \end{equation}
        Letting $j \to +\infty$, we get $u \geq \phi+\psi$. Hence, $u \in \NNP(\phi)$, $u \not\equiv -\infty$ and $P[u]=\phi.$\\
         We now prove that $u$ is a solution to equation \eqref{NPMAE}. Indeed, for every $j \geq 1$, we have
         \begin{equation}\label{NPuj=Fuj}
             \NP{u_j}=F(u_j, \cdot)d\mu_j.
         \end{equation}
        Fix $j_0 \in \mathbb{N}^*$, by the monotonicity of $F$ and the sequence  $\{\Omega_j\}_j$, we obtain
        $$\NP{u_j}=F(u_j, \cdot)d\mu_j \geq F(u, \cdot)d\mu_{j_0},$$
        for all $j \geq j_0 \geq 1$. Thus, since  $u=\lim\limits_{j \to +\infty}u_j$ is not identically $- \infty$, it follows from \cite[Lemma 3.2]{DDLP} that
        $$\NP{u} \geq F(u, \cdot)d\mu_{j_0},$$
        for every  fixed $j_0$. Letting  $j_0 \to +\infty$, it implies that 
        $$\NP{u} \geq \mathbb{1}_{\bigcup\limits_{j_0\geq 1}\Om_{j_0}} F(u, \cdot)d\mu.$$
        Moreover, since $\bigcup_{j_0\geq 1}\Om_{j_0}=\Om \setminus\{\phi+\psi=-\infty\}$ and $\mu$ be a Borel measure does not charge  pluripolar sets, we conclude that
        \begin{equation}\label{NP u > F u}
            \NP{u} \geq  F(u, \cdot)d\mu.
        \end{equation}
        To obtain the reverse inequality, fix $j_0 \in \mathbb{N}^*$. By \eqref{NPuj=Fuj}, we have
        $$\mathbb{1}_{\Om_{j_0}}\NP{u_j}=F(u_j, \cdot)d\mu_{j_0} \leq F(u_{j_0}, \cdot)d\mu_{j_0}.$$ 
        Again, since $\lim\limits_{j \to +\infty}u_j \not\equiv - \infty$, it follows from \cite[Lemma 3.3]{DDLP} that
        $$\mathbb{1}_{\Om_{j_0}}\NP{u} \leq F(u_{j_0}, \cdot)d\mu_{j_0},$$
        for every fixed $j_0 \in \N^*$. Letting $j_0 \to +\infty$ and using the continuity of $F$ in the first variable, we obtain
        $$\mathbb{1}_{\bigcup\limits_{j_0\geq 1}\Om_{j_0}}\NP{u}\leq \mathbb{1}_{\bigcup\limits_{j_0\geq 1}\Om_{j_0}}F(u, \cdot)d\mu.$$
        In addition, since $\Om$ differs from $\bigcup_{j_0\geq 1}\Om_{j_0}$ by a pluripolar set $\{\phi+\psi=-\infty\}$  and both Borel measures $\mu$ and  the non-pluripolar measure does not charge  pluripolar sets, we conclude that
        \begin{equation}\label{NP u < F u}
            \NP{u} \leq F(u, \cdot)d\mu.
        \end{equation}
        Combining inequalities \eqref{NP u > F u} and \eqref{NP u < F u}, we obtain
        $$\NP{u}=F(u, \cdot)d\mu,$$
        which means that $u$ is a solution to the equation \eqref{NPMAE}, as desired. The uniqueness of $u$ follows from \cite[Corollary 5.7]{DDLP}.
        
        The proof is completed.
\end{proof}
\noindent
	\textbf{Acknowledgments.}
The authors would like to thank Prof. Pham Hoang Hiep and Assoc. Prof. Do Hoang Son for their valuable comments.

\end{document}